\documentclass[reqno]{amsart}%
\usepackage{amssymb,mathtools,mathrsfs}
\usepackage{ifpdf}
\ifpdf
 \usepackage[hyperindex]{hyperref}
\else
 \expandafter\ifx\csname dvipdfm\endcsname\relax
 \usepackage[hypertex,hyperindex]{hyperref}%
 \else
 \usepackage[dvipdfm,hyperindex]{hyperref}%
 \fi
\fi
\allowdisplaybreaks[4]
\numberwithin{equation}{section}
\theoremstyle{plain}
\newtheorem{thm}{Theorem}[section]
\theoremstyle{remark}
\newtheorem{rem}{Remark}[section]
\newcommand{\td}{\textup{d}}

\begin{document}

\title[Identities involving central binomial coefficients]
{Several combinatorial identities derived from series expansions of powers of arcsine}

\author[F. Qi]{Feng Qi}
\address{Institute of Mathematics, Henan Polytechnic University, Jiaozuo 454010, Henan, China; College of Mathematics, Inner Mongolia University for Nationalities, Tongliao 028043, Inner Mongolia, China; School of Mathematical Sciences, Tianjin Polytechnic University, Tianjin 300387, China}
\email{\href{mailto: F. Qi <qifeng618@gmail.com>}{qifeng618@gmail.com}, \href{mailto: F. Qi <qifeng618@hotmail.com>}{qifeng618@hotmail.com}, \href{mailto: F. Qi <qifeng618@qq.com>}{qifeng618@qq.com}}
\urladdr{\url{https://qifeng618.wordpress.com}, \url{https://orcid.org/0000-0001-6239-2968}}

\author{Chao-Ping Chen}
\address{School of Mathematics and Informatics, Henan Polytechnic University, Jiaozuo 454010, Henan, China}
\email{chenchaoping@sohu.com}

\author[D. Lim]{Dongkyu Lim}
\address{Department of Mathematics Education, Andong National University, Andong 36729, Republic of Korea}
\email{\href{mailto: D. Lim <dgrim84@gmail.com>}{dgrim84@gmail.com}, \href{mailto: D. Lim <dklim@andong.ac.kr>}{dklim@andong.ac.kr}}
\urladdr{\url{http://orcid.org/0000-0002-0928-8480}}

\dedicatory{Dedicated to people facing and battling COVID-19}

\begin{abstract}
In the paper, with the aid of the series expansions of the square or cubic of the arcsine function, the authors establish several possibly new combinatorial identities containing the ratio of two central binomial coefficients which are related to the Catalan numbers in combinatorial number theory.
\end{abstract}



\thanks{This paper was typeset using\AmS-\LaTeX}

\maketitle
\tableofcontents

\section{Introduction}
The sequence of central binomial coefficients $\binom{2n}{n}$ for $n\ge0$ is classical, simple, and elementary. This sequence has attracted many mathematicians who have published a number of papers such as~\cite{Alzer-Nagy-Integers-2020, Boyad-JIS-2012, John-Maxwell-Rocky-2019, Chen-HW-IJS-2016, Gar, Gu-Guo-BAustMS-2020, Jovan-Mikic-2020, Spivey-art-2019, Sprugnoli-INT-2006} and closely related references therein. It is worth to mentioning that, the integral representation
\begin{equation*}
\binom{2n}{n}=\frac{1}{\pi}\int_{0}^{\infty}\frac{1}{(1/4+s^2)^{n+1}}\td s
\end{equation*}
was derived in~\cite[Section~4.2]{Catalan-Gen-Int-Formula.tex}.
\par
In this paper, with the help of the power series expansion
\begin{equation}\label{arcsin-power-series}
\arcsin x=\sum_{\ell=0}^{\infty}\frac{1}{2^{2\ell}}\binom{2\ell}{\ell}\frac{x^{2\ell+1}}{2\ell+1}, \quad |x|<1,
\end{equation}
see~\cite[4.4.40]{abram} and~\cite[p.~121, 6.41.1]{Adams-Hippisley-Smithsonian1922}, the series expansion
\begin{equation}\label{Lehmer-Monthly-1985-arcsin-square-expan}
(\arcsin x)^2=\frac{1}{2}\sum_{\ell=1}^{\infty}\frac{(2x)^{2\ell}}{\ell^2\binom{2\ell}{\ell}}, \quad |x|<1,
\end{equation}
which or its variants can be found in~\cite[p.~122, 6.42.1]{Adams-Hippisley-Smithsonian1922}, \cite[pp.~262--263, Proposition~15]{Berndt-Ramanujan-B-I}, \cite[pp.~50--51 and p.~287]{Borwein-Bailey-Girgensohn-Experim-2004}, \cite[p.~384]{Borwein-2-book-87}, \cite[Lemma~2]{Chen-CP-ITSF-2012}, \cite[pp.~88--90]{Edwards-1982-DC}, \cite[p.~61, 1.645]{Gradshteyn-Ryzhik-Table-8th}, \cite[p.~308]{Kalmykov-Sheplyakov-2001}, \cite[p.~453]{Lehmer-Monthly-1985}, \cite[Section~6.3]{Catalan-Int-Surv.tex}, \cite[p.~59, (2.56)]{Wilf-GF-2006-3rd}, or~\cite[p.~676, (2.2)]{Zhang-Chen-JMI-2020},
and the power series expansion
\begin{equation}\label{arcsin-cubic-series-expan}
(\arcsin x)^3=3!\sum_{\ell=0}^{\infty}[(2\ell+1)!!]^2\Biggl[\sum_{k=0}^{\ell}\frac{1}{(2k+1)^2}\Biggr]\frac{x^{2\ell+3}}{(2\ell+3)!}, \quad |x|<1,
\end{equation}
which or its variants can be found in~\cite[p.~122, 6.42.2]{Adams-Hippisley-Smithsonian1922}, \cite[pp.~262--263, Proposition~15]{Berndt-Ramanujan-B-I}, \cite[p.~188, Example~1]{Bromwich-1908}, \cite[pp.~88--90]{Edwards-1982-DC}, \cite[p.~61, 1.645]{Gradshteyn-Ryzhik-Table-8th}, or~\cite[p.~308]{Kalmykov-Sheplyakov-2001}, we will establish several identities involving the product $\binom{2k}{k}\binom{2(n-k)}{n-k}$ or the ratio $\frac{\binom{2k}{k}}{\binom{2(n-k+1)}{n-k+1}}$ of two central binomial coefficients $\binom{2k}{k}$ and $\binom{2(n-k)}{n-k}$ for $0\le k\le n$.

\section{Two alternative proofs of a known combinatorial identity}

In this section, by means of the series expansions~\eqref{arcsin-power-series} and~\eqref{Lehmer-Monthly-1985-arcsin-square-expan}, we give two alternative proofs of a known combinatorial identity.

\begin{thm}[{\cite[p.~77, (3.96)]{Sprugnoli-Gould-2006}}]\label{Catalan-Chen-ID-thm}
For $n\ge0$, we have
\begin{equation}\label{Catalan-Chen-ID}
\sum_{k=0}^{n}\frac{1}{2k+1}\binom{2k}{k}\binom{2(n-k)}{n-k}=\frac{2^{4n}}{(2n+1)\binom{2n}{n}}.
\end{equation}
\end{thm}

\begin{proof}[First proof]
From~\eqref{arcsin-power-series}, it follows that
\begin{equation*}
\frac{1}{2}\arcsin(2x)=\sum_{k=0}^{\infty}\frac{1}{2k+1}\binom{2k}{k}x^{2k+1},\quad |x|<\frac{1}{2}
\end{equation*}
and, by differentiation,
\begin{equation*}
\frac{1}{\sqrt{1-4x^2}\,}=\sum_{k=0}^{\infty}\binom{2k}{k}x^{2k},\quad |x|<\frac{1}{2}.
\end{equation*}
Therefore, we obtain
\begin{equation}
\begin{aligned}\label{arcsin-sqrt-cauchy-prod}
\frac{\arcsin(2x)}{2x}\frac{1}{\sqrt{1-4x^2}\,}
&=\Biggl[\sum_{k=0}^{\infty}\frac{1}{2k+1}\binom{2k}{k}x^{2k}\Biggr] \Biggl[\sum_{k=0}^{\infty}\binom{2k}{k}x^{2k}\Biggr]\\
&=\sum_{n=0}^{\infty}\Biggl[\sum_{k=0}^n \frac{1}{2k+1}\binom{2k}{k}\binom{2(n-k)}{n-k}\Biggr]x^{2n}.
\end{aligned}
\end{equation}
On the other hand, by virtue of the series expansion~\eqref{Lehmer-Monthly-1985-arcsin-square-expan}, we acquire
\begin{equation}
\begin{aligned}\label{arcsin-sqrt-squre-expan}
\frac{\arcsin(2x)}{2x}\frac{1}{\sqrt{1-4x^2}\,}
&=\frac{1}{8x}\frac{\td}{\td x}\bigl([\arcsin(2x)]^2\bigr)\\
&=\frac{1}{8x}\frac{\td}{\td x}\sum_{n=0}^{\infty}\frac{2^{2n+1}(n!)^2}{(2n+2)!}(2x)^{2n+2}\\
&=\frac{1}{8x}\sum_{n=0}^{\infty}\frac{2^{2n+2}(n!)^2}{(2n+1)!}(2x)^{2n+1}\\
&=\sum_{n=0}^{\infty}\frac{2^{4n}(n!)^2}{(2n+1)!}x^{2n}.
\end{aligned}
\end{equation}
Comparing~\eqref{arcsin-sqrt-cauchy-prod} with~\eqref{arcsin-sqrt-squre-expan} and equating coefficients of $x^{2n}$, we obtain
\begin{equation*}
\sum_{k=0}^n \frac{1}{2k+1}\binom{2k}{k}\binom{2(n-k)}{n-k}
=\frac{2^{4n}(n!)^2}{(2n+1)!}
=\frac{2^{4n}}{(2n+1)\binom{2n}{n}}.
\end{equation*}
The identity~\eqref{Catalan-Chen-ID} is thus proved.
The first proof of Theorem~\ref{Catalan-Chen-ID-thm} is complete.
\end{proof}

\begin{proof}[Second proof]
Differentiating on both sides of~\eqref{Lehmer-Monthly-1985-arcsin-square-expan} and rearranging give
\begin{equation}\label{Lehmer-Monthly-1985-Gen}
\frac{2x\arcsin x}{\sqrt{1-x^2}\,} =\sum_{\ell=1}^{\infty}\frac{(2x)^{2\ell}}{\ell\binom{2\ell}{\ell}}, \quad |x|<1,
\end{equation}
which or its variants can also be found in~\cite[p.~122, 6.42.5]{Adams-Hippisley-Smithsonian1922}, \cite[p.~384]{Borwein-2-book-87}, \cite[p.~161]{Bradley-Ramanujan-1999}, \cite[p.~452, Theorem]{Lehmer-Monthly-1985}, and~\cite[Section~6.3, Theorem~21, Sections~8 and~9]{Catalan-Int-Surv.tex}.
Replacing $x$ by $2x$ in~\eqref{Lehmer-Monthly-1985-Gen} and rearranging yield
\begin{equation}
\begin{aligned}\label{Lehmer-Monthly-1985-expan}
\frac{\arcsin(2x)}{2x}\frac{1}{\sqrt{1-4x^2}\,}
&=\frac{1}{8x^2}\frac{4x\arcsin(2x)}{\sqrt{1-4x^2}\,}\\
&=\frac{1}{8x^2}\sum_{\ell=1}^{\infty}\frac{(4x)^{2\ell}}{\ell\binom{2\ell}{\ell}}\\
&=\frac{1}{8x^2}\sum_{n=0}^{\infty}\frac{(4x)^{2(n+1)}}{(n+1)\binom{2(n+1)}{n+1}}\\
&=\sum_{n=0}^{\infty}\frac{2^{4n+1}}{(n+1)\binom{2(n+1)}{n+1}}x^{2n}
\end{aligned}
\end{equation}
for $|x|<\frac12$. Comparing~\eqref{arcsin-sqrt-cauchy-prod} with~\eqref{Lehmer-Monthly-1985-expan} and equating coefficients of $x^{2n}$, we obtain
\begin{equation*}
\sum_{k=0}^n \frac{1}{2k+1}\binom{2k}{k}\binom{2(n-k)}{n-k}
=\frac{2^{4n+1}}{(n+1)\binom{2(n+1)}{n+1}}
=\frac{2^{4n}}{(2n+1)\binom{2n}{n}}.
\end{equation*}
The identity~\eqref{Catalan-Chen-ID} is proved again.
The second proof of Theorem~\ref{Catalan-Chen-ID-thm} is complete.
\end{proof}

\section{Three possibly new combinatorial identities}

In this section, by virtue of those three series expansions~\eqref{arcsin-power-series}, \eqref{Lehmer-Monthly-1985-arcsin-square-expan}, and~\eqref{arcsin-cubic-series-expan}, we establish three possibly new combinatorial identities involving the ratio $\frac{\binom{2k}{k}}{\binom{2(n-k+1)}{n-k+1}}$ in terms of the trigamma function $\psi'\bigl(n+\frac{3}{2}\bigr)$, where $\psi(x)$ is the digamma function defined by the logarithmic derivative $\psi(x)=[\ln\Gamma(x)]'=\frac{\Gamma'(x)}{\Gamma(x)}$ of the classical Euler gamma function
\begin{equation*}
\Gamma(z)=\int_0^{\infty}t^{z-1}e^{-t}\td t, \quad \Re(z)>0.
\end{equation*}
For more information on the gamma function $\Gamma(x)$ and polygamma functions $\psi^{(k)}(x)$ for $k\ge0$, please refer to~\cite[pp.~255--293, Chapter~6]{abram} or the papers~\cite{convex-bern-GF-S.tex, Ouimet-LCM-BKMS.tex} and closely related references therein.

\begin{thm}\label{3identities-ratio=thm}
For $n\ge0$, we have
\begin{gather}\label{CBC-ratio-thm1}
\sum_{k=0}^{n}\frac{1}{2^{4k}(2k+1)(n-k+1)^2} \frac{\binom{2k}{k}}{\binom{2(n-k+1)}{n-k+1}}
=\frac{3[(2n+1)!!]^2}{2^{2n+3}(2n+3)!} \biggl[\pi^2-2\psi'\biggl(n+\frac{3}{2}\biggr)\biggr],\\
\sum_{k=0}^{n}\frac{1}{2^{4k}(n-k+1)^2}\frac{\binom{2k}{k}}{\binom{2(n-k+1)}{n-k+1}}
=\frac{[(2n+1)!!]^2}{2^{2n+3}(2n+2)!}\biggl[\pi^2-2\psi'\biggl(n+\frac{3}{2}\biggr)\biggr],
\label{CBC-ratio-thm2}
\end{gather}
and
\begin{equation}\label{CBC-ratio-thm3}
\sum_{k=0}^{n}\frac{1}{2^{4k}(2k+1)(n-k+1)} \frac{\binom{2k}{k}}{\binom{2(n-k+1)}{n-k+1}}
=\frac{[(2n+1)!!]^2}{2^{2n+3}(2n+2)!} \biggl[\pi^2-2\psi'\biggl(n+\frac{3}{2}\biggr)\biggr].
\end{equation}
\end{thm}

\begin{proof}
Differentiating on both sides of~\eqref{arcsin-cubic-series-expan} gives
\begin{equation}\label{arcsin-cubic-ser-expan-der}
\frac{(\arcsin x)^2}{\sqrt{1-x^2}\,}=2!\sum_{n=0}^{\infty}[(2n+1)!!]^2\Biggl[\sum_{k=0}^{n}\frac{1}{(2k+1)^2}\Biggr]\frac{x^{2n+2}}{(2n+2)!}, \quad |x|<1.
\end{equation}
On the other hand, we have
\begin{align*}
(\arcsin x)^3&=(\arcsin x)(\arcsin x)^2\\
&=\Biggl[\sum_{k=0}^{\infty}\frac{1}{2^{2k}}\binom{2k}{k}\frac{x^{2k+1}}{2k+1}\Biggr] \Biggl[\frac{1}{2}\sum_{m=1}^{\infty}\frac{(2x)^{2m}}{m^2\binom{2m}{m}}\Biggr]\\
&=\frac{x^3}{2}\Biggl[\sum_{k=0}^{\infty}\frac{1}{(2k+1)2^{2k}}\binom{2k}{k}x^{2k}\Biggr] \Biggl[\sum_{k=0}^{\infty}\frac{2^{2k}}{(k+1)^2\binom{2(k+1)}{k+1}}x^{2k}\Biggr]\\
&=\frac{x^3}{2}\sum_{n=0}^{\infty}\Biggl[\sum_{k=0}^{n}\frac{1}{(2k+1)2^{2k}}\binom{2k}{k} \frac{2^{2(n-k)}}{(n-k+1)^2\binom{2(n-k+1)}{n-k+1}}\Biggr]x^{2n}\\
&=\sum_{n=0}^{\infty}\Biggl[\sum_{k=0}^{n}\frac{2^{2(n-2k)-1}}{(2k+1)(n-k+1)^2} \frac{\binom{2k}{k}}{\binom{2(n-k+1)}{n-k+1}}\Biggr]x^{2n+3},\\
\frac{(\arcsin x)^2}{\sqrt{1-x^2}\,}&=(\arcsin x)^2\frac{1}{\sqrt{1-x^2}\,}\\
&=\Biggl[\frac{1}{2}\sum_{m=1}^{\infty}\frac{(2x)^{2m}}{m^2\binom{2m}{m}}\Biggr] \Biggl[\sum_{n=0}^{\infty}\frac{1}{2^{2n}}\binom{2n}{n}x^{2n}\Biggr]\\
&=x^2\Biggl[\sum_{k=0}^{\infty}\frac{2^{2k+1}}{(k+1)^2\binom{2(k+1)}{k+1}}x^{2k}\Biggr] \Biggl[\sum_{k=0}^{\infty}\frac{1}{2^{2k}}\binom{2k}{k}x^{2k}\Biggr]\\
&=\sum_{n=0}^{\infty}\Biggl[\sum_{k=0}^{n}\frac{1}{2^{2k}}\binom{2k}{k} \frac{2^{2(n-k)+1}}{(n-k+1)^2\binom{2(n-k+1)}{n-k+1}}\Biggr]x^{2n+2}\\
&=\sum_{n=0}^{\infty}\Biggl[\sum_{k=0}^{n}\frac{2^{2(n-2k)+1}}{(n-k+1)^2}\frac{\binom{2k}{k}}{\binom{2(n-k+1)}{n-k+1}}\Biggr]x^{2n+2},
\end{align*}
and
\begin{align*}
\frac{(\arcsin x)^2}{\sqrt{1-x^2}\,}&=(\arcsin x)\frac{\arcsin x}{\sqrt{1-x^2}\,}\\
&=\Biggl[\sum_{k=0}^{\infty}\frac{1}{2^{2k}}\binom{2k}{k}\frac{x^{2k+1}}{2k+1}\Biggr] \Biggl[\frac{1}{2x}\sum_{m=1}^{\infty}\frac{(2x)^{2m}}{m\binom{2m}{m}}\Biggr]\\
&=\Biggl[\sum_{k=0}^{\infty}\frac{1}{2^{2k}(2k+1)}\binom{2k}{k}x^{2k}\Biggr] \Biggl[\sum_{k=0}^{\infty}\frac{2^{2k+1}}{(k+1)\binom{2(k+1)}{k+1}}x^{2(k+1)}\Biggr]\\
&=\sum_{n=0}^{\infty}\Biggl[\sum_{k=0}^{n}\frac{1}{2^{2k}(2k+1)}\binom{2k}{k} \frac{2^{2(n-k)+1}}{(n-k+1)\binom{2(n-k+1)}{n-k+1}}\Biggr]x^{2n+2}\\
&=\sum_{n=0}^{\infty}\Biggl[\sum_{k=0}^{n}\frac{2^{2(n-2k)+1}}{(2k+1)(n-k+1)} \frac{\binom{2k}{k}}{\binom{2(n-k+1)}{n-k+1}}\Biggr]x^{2n+2},
\end{align*}
where we used the power series expansions~\eqref{arcsin-power-series}, \eqref{Lehmer-Monthly-1985-arcsin-square-expan}, and~\eqref{Lehmer-Monthly-1985-Gen}. Equating the above three power series expansions with series expansions~\eqref{arcsin-cubic-series-expan} and~\eqref{arcsin-cubic-ser-expan-der} respectively reveals
\begin{gather}\label{CBC-ratio-1}
\frac{3![(2n+1)!!]^2}{(2n+3)!}\Biggl[\sum_{k=0}^{n}\frac{1}{(2k+1)^2}\Biggr]
=\sum_{k=0}^{n}\frac{2^{2(n-2k)-1}}{(2k+1)(n-k+1)^2} \frac{\binom{2k}{k}}{\binom{2(n-k+1)}{n-k+1}},\\
\frac{2![(2n+1)!!]^2}{(2n+2)!}\Biggl[\sum_{k=0}^{n}\frac{1}{(2k+1)^2}\Biggr]
=\sum_{k=0}^{n}\frac{2^{2(n-2k)+1}}{(n-k+1)^2}\frac{\binom{2k}{k}}{\binom{2(n-k+1)}{n-k+1}},
\label{CBC-ratio-2}
\end{gather}
and
\begin{equation}\label{CBC-ratio-3}
\frac{2![(2n+1)!!]^2}{(2n+2)!}\Biggl[\sum_{k=0}^{n}\frac{1}{(2k+1)^2}\Biggr]
=\sum_{k=0}^{n}\frac{2^{2(n-2k)+1}}{(2k+1)(n-k+1)} \frac{\binom{2k}{k}}{\binom{2(n-k+1)}{n-k+1}}.
\end{equation}
From the formula
\begin{equation*}
\psi'\biggl(\frac{1}{2}+n\biggr)=\frac{\pi^2}{2}-4\sum_{k=1}^{n}\frac{1}{(2k-1)^2}, \quad n\in\mathbb{N}
\end{equation*}
in~\cite[p.~914, 8.366]{Gradshteyn-Ryzhik-Table-8th}, we derive that
\begin{equation}\label{2k+1-harm-trigamma}
\sum_{k=0}^{n}\frac{1}{(2k+1)^2}=\frac{1}{8} \biggl[\pi^2-2 \psi'\biggl(n+\frac{3}{2}\biggr)\biggr].
\end{equation}
Substituting the formula~\eqref{2k+1-harm-trigamma} into~\eqref{CBC-ratio-1}, \eqref{CBC-ratio-2}, and~\eqref{CBC-ratio-3} and simplifying lead to three identities~\eqref{CBC-ratio-thm1}, \eqref{CBC-ratio-thm2}, and~\eqref{CBC-ratio-thm3} respectively.
The proof of Theorem~\ref{3identities-ratio=thm} is thus complete.
\end{proof}

\section{Remarks}

Finally, we list several remarks on our main results and related stuffs.

\begin{rem}
The identity~\eqref{Catalan-Chen-ID} in Theorem~\ref{Catalan-Chen-ID-thm} can be regarded as a couple of the identity
\begin{equation}\label{Monthly-P11897-ID-final}
\sum_{k=0}^{n}\frac{1}{k+1}\binom{2k}{k}\binom{2(n-k)}{n-k}
=\binom{2n+1}{n}, \quad n\ge0,
\end{equation}
which is a special case of the identity~\cite[p.~77, (3.95)]{Sprugnoli-Gould-2006}. Moreover, the identity
\begin{equation}\label{Monthly-P11897-ID}
\sum_{\substack{k+\ell=n,\\ k\ge0,\ell\ge0}}\frac{1}{k+1}\binom{2k}{k}\binom{2(\ell+1)}{\ell+1}=2\binom{2n+2}{n},\quad n\ge0,
\end{equation}
which has been proved in~\cite{Monthly-P11897} by three alternative and different methods, is an equivalence of the identity~\eqref{Monthly-P11897-ID-final}. This equivalence can be demonstrated as follows.
\par
The identity~\eqref{Monthly-P11897-ID} can be rearranged as
\begin{equation*}
\sum_{k=0}^n\frac{1}{k+1}\binom{2k}{k}\binom{2(n-k+1)}{n-k+1}=2\binom{2n+2}{n}
\end{equation*}
which is equivalent to
\begin{equation*}
\sum_{k=0}^{n+1}\frac{1}{k+1}\binom{2k}{k}\binom{2(n-k+1)}{n-k+1}
=2\binom{2n+2}{n}+\frac{1}{n+2}\binom{2(n+1)}{n+1}=\binom{2n+3}{n+1},
\end{equation*}
where we used $\binom{0}{0}=1$. Replacing $n+1$ by $n$ in the last identity leads to the identity~\eqref{Monthly-P11897-ID-final}.
\end{rem}

\begin{rem}
Closely related to central binomial coefficients $\binom{2n}{n}$, the Catalan numbers
\begin{equation}\label{Catalan-1Exp}
C_n=\frac{1}{n+1}\binom{2n}{n}
\end{equation}
in combinatorial number theory have attracted many mathematicians who have published several monographs~\cite{Grimaldi-B2012, Koshy-B-2009, Roman-Catalan-B2015, Stanley-Catalan-2015} and a number of papers such as~\cite{Catalan-Li-Qi-Kouba.tex, Mathematics-129120.tex, Ars-Comb-Catalan-Asymp-Qi.tex, Fuss-Cat-Qi-Cer.tex, Catalan-GF-Plus.tex, Filomat-5244.tex, JAAC961.tex, 195-2017-JOCAAA.tex, Kims-Rus-Catalan.tex, AADM-3031.tex}.
\par
The second conclusion~(b) in~\cite[Lemma~2]{Alzer-Nagy-Integers-2020} reads that
\begin{equation}\label{Alzer-Nagy-Lem2(b)}
\sum_{k=0}^n B_k C_{n-k}=\frac12 B_{n+1},
\end{equation}
where $B_n=\binom{2n}n$.
Rewriting the sum in~\eqref{Alzer-Nagy-Lem2(b)} as $\sum_{k=0}^n B_{n-k}C_k$ and substituting $\binom{2n-2k}{n-k}$ and $\frac{1}{k+1}\binom{2k}{k}$ for $B_{n-k}$ and  $C_k$ result in
\begin{equation*}
\sum_{k=0}^n \frac{1}{k+1}\binom{2k}{k}\binom{2(n-k)}{n-k}C_k=\frac12\binom{2(n+1)}{n+1}=\binom{2n+1}{n}
\end{equation*}
which is the same as the identity~\eqref{Monthly-P11897-ID-final}.
\par
By the way, the combinatorial proof of the identity~\eqref{Alzer-Nagy-Lem2(b)} in~\cite[Lemma~2]{Alzer-Nagy-Integers-2020} is longer than the combinatorial proof of the identity~\eqref{Monthly-P11897-ID} in~\cite{Monthly-P11897}, while its equivalent identities~\eqref{Monthly-P11897-ID-final} and~\eqref{Monthly-P11897-ID} were proved analytically in~\cite{Monthly-P11897} and~\cite[p.~77, (3.95)]{Sprugnoli-Gould-2006}.
\end{rem}

\begin{rem}
By the formula~\eqref{Catalan-1Exp}, we can rewritten the identity~\eqref{Monthly-P11897-ID-final} and those in Theorem~\ref{Catalan-Chen-ID-thm} and Theorem~\ref{3identities-ratio=thm} as
\begin{align*}
\sum_{k=0}^{n}(n-k+1)C_kC_{n-k}&=\binom{2n+1}{n},\\
\sum_{k=0}^{n}\frac{(k+1)(n-k+1)}{2k+1}C_kC_{n-k}&=\frac{2^{4n}}{(2n+1)(n+1)C_n},\\
\sum_{k=0}^{n}\frac{n-k+2}{2^{4k}(k+1)(2k+1)(n-k+1)^2}\frac{C_k}{C_{n-k+1}}
&=\frac{3[(2n+1)!!]^2}{2^{2n}(2n+3)!}\sum_{k=0}^{n}\frac{1}{(2k+1)^2},\\
\sum_{k=0}^{n}\frac{n-k+2}{2^{4k}(k+1)(n-k+1)^2}\frac{C_k}{C_{n-k+1}}
&=\frac{[(2n+1)!!]^2}{2^{2n}(2n+2)!}\sum_{k=0}^{n}\frac{1}{(2k+1)^2},
\end{align*}
and
\begin{equation*}
\sum_{k=0}^{n}\frac{n-k+2}{2^{4k}(k+1)(2k+1)(n-k+1)}\frac{C_k}{C_{n-k+1}}
=\frac{[(2n+1)!!]^2}{2^{2n}(2n+2)!}\sum_{k=0}^{n}\frac{1}{(2k+1)^2}
\end{equation*}
respectively. For more information on series involving the Catalan numbers $C_n$, please refer to the paper~\cite{MTJPAM-D-19-00007.tex} and closely related references therein.
\end{rem}

\begin{rem}
This paper is a revised version of the arXiv preprint~\cite{Catalan-Chen-Id.tex}.
\end{rem}

\subsection*{Acknowledgements}
The authors thank
\begin{enumerate}
\item
Nguyen Xuan Tho (Hanoi University of Science and Technology, Vietnam) for providing the reference~\cite{Monthly-P11897} on 9 November 2020;
\item
Mikhail Kalmykov (kalmykov.mikhail@googlemail.com) for providing the papers~\cite{Davydychev-Kalmykov-2004, Kalmykov-Sheplyakov-lsjk-2005, Kalmykov-Sheplyakov-2001} on 9 January 2021;
\item
anonymous referees for pointing out or reminding of the papers~\cite{Alzer-Nagy-Integers-2020, Sprugnoli-Gould-2006, Witula-IJPAM-2013}.
\end{enumerate}

\end{document}